\documentclass{amsart}
\usepackage[final]{epsfig}
\usepackage{graphics}
\usepackage{float}
\usepackage{amsmath}
\usepackage{amsfonts}
\usepackage{latexsym}
\usepackage{amssymb}
\usepackage{graphicx}
\usepackage{url}
\usepackage{epstopdf}
\usepackage{verbatim}
\usepackage[margin=1in]{geometry}
\usepackage{amsthm}
\usepackage{tikz}
\usepackage{caption}
\usepackage{tikz-cd}
\usepackage{enumitem}

\makeatletter
\newtheorem*{rep@theorem}{\rep@title}

\definecolor{orange}{rgb}{1,0.5,0}

\newcommand{\newreptheorem}[2]{%
\newenvironment{rep#1}[1]{%
 \def\rep@title{#2 \ref{##1}}%
 \begin{rep@theorem}}%
 {\end{rep@theorem}}}
\makeatother

\newtheorem{lemma}{Lemma}[section]
\newtheorem{proposition}[lemma]{Proposition}

\newtheorem{remark}[lemma]{Remark}

\newtheorem{theorem}[lemma]{Theorem}
\newtheorem{definition}[lemma]{Definition}
\newtheorem{corollary}[lemma]{Corollary}

\newtheorem*{theorem*}{Theorem}
\newreptheorem{theorem}{Theorem}

\newcommand{\I}{{\mathcal I}}

\newcommand{\Fl}{{\mathsf{Fl}}}

\newcommand{\R}{{\mathbb R}}

\DeclareMathOperator{\conv}{Conv}

\DeclareMathOperator{\Perm}{Perm}
\DeclareMathOperator{\G}{Br}
\DeclareMathOperator{\Br}{Br}

\begin{document}

\title{A positive Grassmannian analogue of the permutohedron}

\author{Lauren K. Williams}
\date{\today}
\thanks{The author was
partially supported by an NSF CAREER award DMS-1049513.}
\address{Department of Mathematics, University of California,
Berkeley, CA 94720-3840}
\email{williams@math.berkeley.edu}
\subjclass[2000]{}

\begin{abstract}
The classical permutohedron $\Perm_n$ is the convex hull of the 
points $(w(1),\dots,w(n))\in \R^n$ where $w$ ranges over all
permutations in the symmetric group $S_n$.
This polytope has many beautiful properties -- for example
it provides a way to visualize the weak Bruhat order:
if we orient
the permutohedron so that 
the longest permutation $w_0$ is at the ``top" 
and the identity $e$ is at the ``bottom,"
then 
the one-skeleton of $\Perm_n$ is the Hasse diagram of the weak 
Bruhat order.  
Equivalently, the paths from $e$ to $w_0$ along
the edges of $\Perm_n$ are in bijection 
with the reduced decompositions of $w_0$.
Moreover, the two-dimensional faces of the permutohedron correspond to 
braid and commuting moves, which by the Tits Lemma, connect any two 
reduced expressions of $w_0$.

In this note we introduce some polytopes $\G_{k,n}$ (which we call
\emph{bridge polytopes}) which provide a 
positive Grassmannian analogue of the permutohedron.  
In this setting, \emph{BCFW bridge decompositions} of 
\emph{reduced plabic graphs} play the role of 
reduced decompositions.  We define $\G_{k,n}$ and explain
how paths along its edges encode 
BCFW bridge decompositions of 
the longest element $\pi_{k,n}$ in the \emph{circular Bruhat order}. 
We also show that two-dimensional faces of $\G_{k,n}$ correspond
to certain local moves for plabic graphs, which by a result of Postnikov \cite{Postnikov},
connect any two reduced plabic graphs associated to $\pi_{k,n}$.
All of these results can be generalized to the positive parts of Schubert cells.
A useful tool in our proofs is the fact that 
our polytopes are isomorphic to certain
\emph{Bruhat interval 
polytopes}. Conversely, our results on bridge polytopes
allow us to deduce some corollaries about the structure of Bruhat interval 
polytopes.
\end{abstract}

\maketitle
\setcounter{tocdepth}{1}
\tableofcontents

\section{Introduction}
The \emph{totally nonnegative part} $(Gr_{k,n})_{\geq 0}$ of the real Grassmannian
$Gr_{k,n}$ is the locus where all Pl\"ucker coordinates are 
non-negative \cite{Lusztig3, Rietsch1, Postnikov}.
Postnikov initiated the combinatorial study of $(Gr_{k,n})_{\geq 0}$: 
he showed that if one stratifies
the space based on which Pl\"ucker coordinates are positive and which are zero, one gets 
a cell decomposition.  The cells are in bijection with several families of 
combinatorial objects, including decorated permutations, and 
\emph{equivalence classes of reduced plabic graphs}.  
Reduced plabic graphs are a certain family of planar bicolored graphs 
embedded in a disk.  These graphs are very interesting:
they can be used to parameterize cells in  $(Gr_{k,n})_{\geq 0}$ \cite{Postnikov} and
to understand the 
\emph{cluster algebra structure} on the Grassmannian \cite{Scott}; they
also arise as  \emph{soliton graphs} associated to the 
KP hierarchy \cite{KW}.  
Most recently reduced plabic graphs and 
the totally non-negative Grassmannian $(Gr_{k,n})_{\geq 0}$ 
have been studied in connection with 
\emph{scattering amplitudes} in $N=4$ super Yang-Mills \cite{Scatt}.  
Motivated by physical considerations,
the authors gave a systematic way for building up a 
reduced plabic graph for a given cell of $(Gr_{k,n})_{\geq 0}$, 
which they called the \emph{BCFW-bridge construction}.

In many ways, reduced plabic graphs behave like reduced decompositions of the symmetric group:
for example, just as any two reduced decompositions of the same permutation can be related by 
braid and commuting moves, any two reduced plabic graphs associated to the same decorated
permutation
can be related via 
certain local moves \cite{Postnikov}.  
The goal of this note is to highlight another way in which reduced plabic graphs
behave like reduced decompositions.  We will introduce a polytope
called the \emph{bridge polytope}, which is a positive Grassmannian analogue of 
the permutohedron in the following sense: just as the one-skeleton of the permutohedron
encodes reduced decompositions of permutations, 
the one-skeleton of the bridge polytope encodes BCFW-bridge decompositions
of reduced plabic graphs.  Moreover, just as the two-dimensional faces of permutohedra
encode braid and commuting moves for reduced decompositions, 
the two-dimensional faces of bridge polytopes
encode local moves for reduced plabic graphs.







\textsc{Acknowledgments:}
The author is grateful to Nima Arkani-Hamed, Jacob Bourjaily, and 
Yan Zhang for 
conversations concerning the BCFW bridge construction.

\section{Background}\label{sec:background}

\subsection{Reduced decompositions, Bruhat order, and the permutohedron}

Recall that the \emph{symmetric group} $S_n$ is the 
group of all permutations on $n$ letters.
If we let $s_i$ denote the simple transposition  exchanging 
$i$ and $i+1$, then $S_n$ is a Coxeter group generated by the
$s_i$ (for $1 \leq i <n$), subject to the relations $s_i^2 = 1$, 
as well as the \emph{commuting relations}
$s_i s_j = s_j s_i$ for $|i-j|\geq 2$ and \emph{braid relations}
$s_i s_{i+1} s_i = s_{i+1} s_i s_{i+1}$.

Given $w\in S_n$, there are many ways to write $w$ 
as a product of  simple transpositions.  
If $w = s_{i_1} \dots s_{i_\ell}$ is a minimal-length such expression,
we call $s_{i_1} \dots s_{i_\ell}$ a \emph{reduced expression} of $w$,
and $\ell$ the \emph{length} $\ell(w)$ of $w$.
The well-known Tits Lemma asserts that any two reduced expressions
can be related by a sequence of braid
and commuting moves if and only if they are reduced expressions for the same
permutation.

There are two important partial orders on $S_n$, both of which are
graded, with rank
function given by the length of the permutation.
The \emph{strong Bruhat order} is the transitive closure
of the cover relation $u \prec v$, where 
$u \prec v$ means that $(i j) u = v$ for $i<j$
and $\ell(v) = \ell(u)+1$.
Here $(i j)$ is the transposition (not necessarily simple) exchanging
$i$ and $j$.
We use the symbol $\preceq$ to indicate the strong Bruhat order.
The strong Bruhat order has many nice combinatorial properties:
it is thin and shellable, and hence is a regular CW poset \cite{Edelman,
Proctor, BW}.
Moreover, this partial order is closely connected to the geometry of the complete flag 
variety $\Fl_n$: it encodes when one Schubert cell is contained 
in the closure of another.  Somewhat less well-known is its connection to 
total positivity.  The totally non-negative part 
$U_{\geq 0}^+$ of the unipotent part of $GL_n$ has a decomposition into
cells, indexed by permutations, and the strong Bruhat order describes
when one cell is contained in the closure of another
\cite{Lusztig3}.

The other partial order associated to $S_n$
is the \emph{(left) weak Bruhat order}.  This
is the transitive closure
of the cover relation $u \lessdot v$, where 
$u \lessdot v$ means that $s_i u = v$ 
and $\ell(v) = \ell(u)+1$. We use the symbol
$\leq$ to indicate the weak Bruhat order.
The weak Bruhat order is 
related to reduced decompositions in the following way:
given any $w\in S_n$, the maximal chains from the identity $e$ to $w$
in the weak Bruhat order are in bijection with reduced decompositions of $w$.
Moreover, the weak order
has the advantage of being conveniently visualized
in terms of a polytope called the permutohedron, as we now describe.

\begin{definition}
The \emph{permutohedron} $\Perm_n$ in $\R^n$ is the convex
hull of the  points 
$\{(z(1),\dots,z(n)) \ \vert \ z\in S_n\}$.
\end{definition}

See Figure \ref{fig:perm} for the permutohedron $\Perm_4$.
We remark that the permutohedron is also connected to the geometry of the 
complete flag variety $\Fl_n$: it is the image 
of $\Fl_n$ under the moment map.

Given a permutation $z\in S_n$, we will often refer to it using the notation
$(z(1),\dots, z(n))$, or 
$(z_1,\dots, z_n)$, where $z_i$ denotes $z(i)$. 

\begin{figure}[h]
\centering
\includegraphics[height=2in]{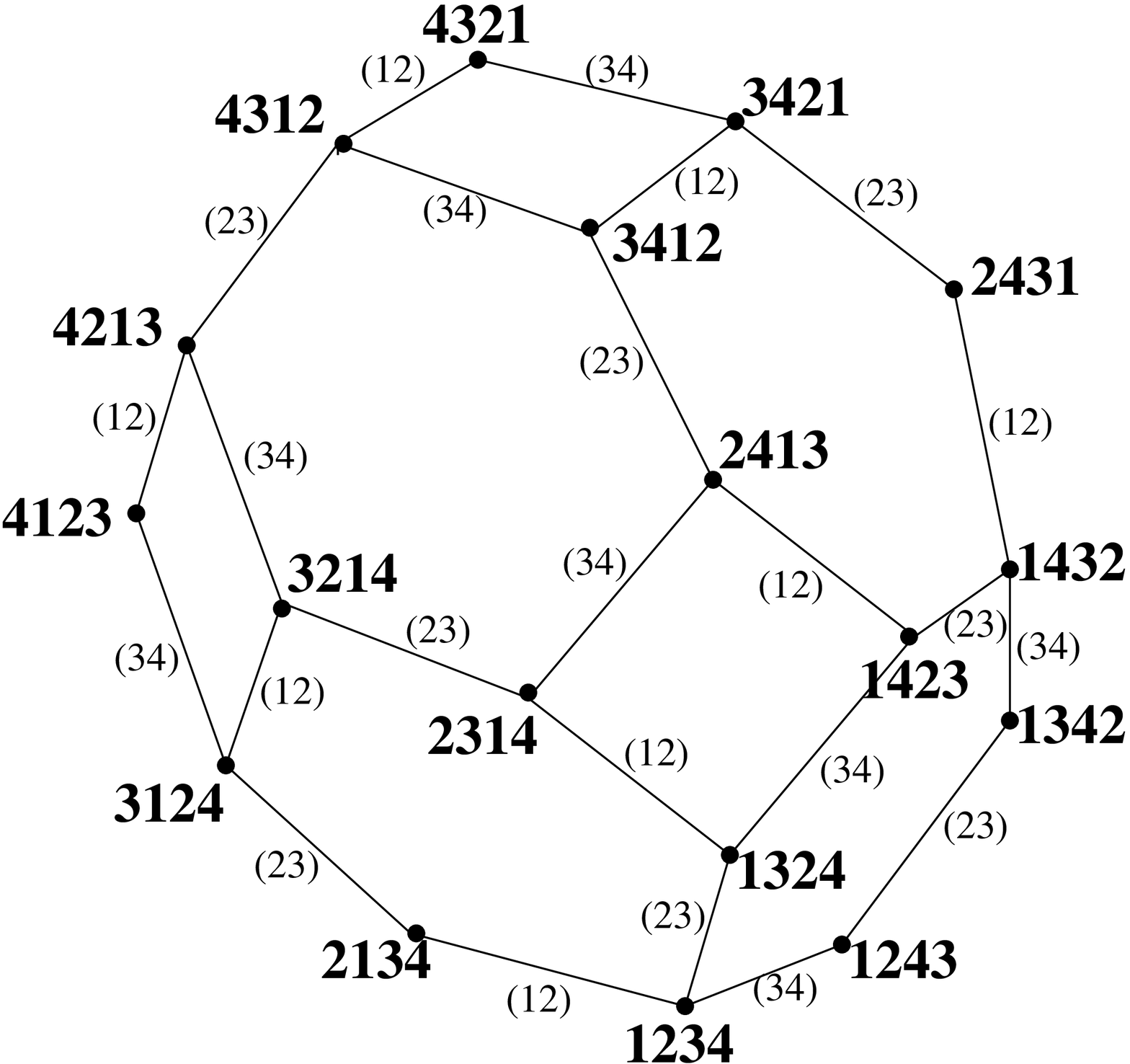}
\caption{The permutohedron $\Perm_4$.  Edges are labeled by 
the values swapped when we go between the permutations labeling
the two vertices.}
\label{fig:perm}
\end{figure}

The following proposition is well-known, see e.g. \cite{OM}.
\begin{proposition}\label{prop:Perm}
Let $u$ and $v$ be permutations in $S_n$ with 
$\ell(u) \leq \ell(v)$.
There is an edge between vertices $u$ and $v$ in 
the permutohedron $\Perm_n$ if and only if 
$v=s_i u$ for some $i$ and $\ell(v) = \ell(u) \pm 1$, in which case
we label the corresponding edge of the permutohedron by $s_i$.
Moreover, the two-dimensional faces of $\Perm_n$ are either squares
(with edges labeled in an alternating fashion by $s_i$ and $s_j$,
where $|i-j| \geq 2$) or hexagons
(with edges labeled in an alternating fashion by 
$s_i$ and $s_{i+1}$ for some $i$).
\end{proposition}

See Figure \ref{fig:perm} for the example of $\Perm_4$.
Proposition \ref{prop:Perm} implies that edges of the permutohedron correspond to 
cover relations in the weak Bruhat order, and the two-dimensional faces correspond to 
braid and commuting relations for reduced words.

Let $w_0$ be the longest permutation in $S_n$.  This permutation
is unique, and can be defined by $w_0(i) = n+1-i$.  
Proposition \ref{prop:Perm} can be equivalently restated as follows.

\begin{proposition}\label{prop:Perm2}
The shortest paths from $e$ to $w_0$ along the one-skeleton 
of the permutohedron 
$\Perm_n$ are in bijection with the reduced decompositions of $w_0$.
\end{proposition}

For example, we can see from Figure \ref{fig:perm} that there
is a path 
$$e \lessdot (2,1,3,4) \lessdot (3,1,2,4) \lessdot (4,1,2,3) 
\lessdot (4,2,1,3) \lessdot (4,3,1,2) \lessdot (4,3,2,1) = w_0$$
in the permutohedron $\Perm_4$.  Reading off the edge labels of this path 
gives rise to the reduced decomposition
$s_1 s_2 s_1 s_3 s_2 s_1$ of $w_0$.

\subsection{The positive Grassmannian, permutations, circular Bruhat order,
and plabic graphs}

In this section we introduced the positive Grassmannian, and some combinatorial
objects associated to it, including permutations and plabic graphs.  As we will see, 
\emph{reduced plabic graphs} are in many ways analogous to reduced decompositions of 
permutations.

The \emph{real Grassmannian} $Gr_{k,n}$ is the space of all
$k$-dimensional subspaces of $\R^n$.  An element of
$Gr_{k,n}$ can be viewed as a full-rank $k\times n$ matrix modulo left
multiplication by nonsingular $k\times k$ matrices.  In other words, two
$k\times n$ matrices represent the same point in $Gr_{k,n}$ 
if and only if they
can be obtained from each other by row operations.
Let $\binom{[n]}{k}$ be the set of all $k$-element subsets of $[n]:=\{1,\dots,n\}$.
For $I\in \binom{[n]}{k}$, let $\Delta_I(A)$
be the {\it Pl\"ucker coordinate}, that is, the maximal minor of the $k\times n$ matrix $A$ located in the column set $I$.
The map $A\mapsto (\Delta_I(A))$, where $I$ ranges over $\binom{[n]}{k}$,
induces the {\it Pl\"ucker embedding\/} $Gr_{k,n}\hookrightarrow \mathbb{RP}^{\binom{n}{k}-1}$.
The \emph{totally non-negative part of the Grassmannian} (sometimes called the 
\emph{positive Grassmannian})
$(Gr_{k,n})_{\geq 0}$ is the subset of $Gr_{k,n}$ such that all
Pl\"ucker coordinates are non-negative.

If we partition $(Gr_{k,n})_{\geq 0}$ into strata based on 
which Pl\"ucker coordinates are positive and which are zero,
we obtain a decomposition into \emph{positroid cells} \cite{Postnikov}.
Postnikov showed that the cells are in bijection with, and naturally
labeled by, several 
families of combinatorial objects, including 
\emph{Grassmann necklaces},
\emph{decorated permutations} and 
\emph{equivalence classes of reduced plabic graphs}.  
He also introduced the \emph{circular Bruhat order} on decorated permutations, 
which describes when one cell is contained in another. 
Like the strong Bruhat order, the circular Bruhat order is 
graded, thin, shellable, and hence is a regular CW poset \cite{Williams}.

\begin{definition}
A \emph{Grassmann necklace} is a sequence
$\I = (I_1,\dots,I_n)$ of subsets $I_r \subset [n]$ such that,
for $i\in [n]$, if $i\in I_i$ then $I_{i+1} = (I_i \setminus \{i\}) \cup \{j\}$,
for some $j\in [n]$; and if $i \notin I_r$ then $I_{i+1} = I_i$.
(Here indices $i$ are taken modulo $n$.)  In particular, we have
$|I_1| = \dots = |I_n|$, which is equal to some $k \in [n]$.  We then say that
$\I$ is a Grassmann necklace of \emph{type} $(k,n)$.
\end{definition}

To construct the Grassmann necklace associated to a positroid cell,
one uses the following construction.

\begin{lemma} \label{lem:necklace}
Define the shifted linear order
$<_i$  by $i <_i i+1 <_i \dots <_i n <_i 1 <_i \dots <_i i-1$.
Given $A\in Gr_{k,n}$,
let $\mathcal I(A) = (I_1,\dots,I_n)$ be the sequence of subsets in
$[n]$ such that, for $i \in [n]$, $I_i$ is the lexicographically
minimal subset of $\binom{[n]}{k}$  with respect to 
the shifted linear order
 such that $\Delta_{I_i}(A) \neq 0$.
Then $\I(A)$ is a Grassmann necklace of type $(k,n)$.
\end{lemma}

\begin{definition}
A \emph{decorated permutation} $\pi$ on $n$ letters is a permutation
on $n$ letters in which fixed points have one of two colors, 
``clockwise" and ``counterclockwise".  A position $i$ of $\pi$
is called a \emph{weak excedance} if $\pi(i) >i$ or 
$\pi(i)= i$ and $i$ is a counterclockwise fixed point.
\end{definition}

\begin{definition}\label{necklace-to-perm}
Given a Grassmann necklace $\mathcal I$, define
a decorated permutation $\pi = \pi(\mathcal I)$ by requiring that
\begin{enumerate}
\item if $I_{i+1} = (I_i \setminus \{i\}) \cup \{j\}$,
for $j \neq i$, then $\pi(j)=i$.
\item if $I_{i+1}=I_i$ and $i \in I_i$ then $\pi(i)=i$ is 
a counterclockwise fixed point.
\item if $I_{i+1}=I_i$ and $i \notin I_i$ then $\pi(i)=i$ is 
a clockwise fixed point.
\end{enumerate}
As before, indices are taken modulo $n$.
\end{definition}

Using the above constructions, one may show the following.

\begin{theorem}\cite{Postnikov}
The cells of $(Gr_{k,n})_{\geq 0}$ are in bijection with 
Grassmann necklaces of type $(k,n)$ and with 
decorated permutations on $n$ letters with exactly $k$
weak excedances.
\end{theorem}

The totally non-negative Grassmannian 
$(Gr_{k,n})_{\geq 0}$ has a unique top-dimensional cell, consisting
of elements where all Pl\"ucker coordinates are strictly positive.
This subset is called the \emph{totally positive Grassmannian}
$(Gr_{k,n})_{>0}$, and it is labeled by the decorated permutation
$\pi_{k,n}:=(n-k+1, n-k+2,\dots, n, 1, 2, \dots, n-k)$.
More generally, there is a nice class of positroid cells called the 
{\it TP} or
{\it totally positive Schubert cells}. 

\begin{definition}\label{def:TPSchubert}
A \emph{TP Schubert cell} is the unique positroid cell of 
greatest dimension which lies in 
the intersection of a usual Schubert cell with $(Gr_{k,n})_{\geq 0}$.
The TP Schubert cells in $(Gr_{k,n})_{\geq 0}$ are in bijection with 
$k$-element subsets $J$ of $[n]=\{1,2,\dots,n\}$.  
To calculate the decorated permutation
associated to the TP Schubert cell indexed by $J$,
write $J = \{j_1 < \dots < j_k\}$, and the complement of $J$ 
as $J^c =\{h_1 < \dots < h_{n-k}\}$.  Then 
the corresponding decorated permutation $\pi = \pi(J)$ is defined by 
$\pi(h_i) = i$ for all $1 \leq i \leq n-k$,
and $\pi(j_i) = n-k+i$ for all $1 \leq i \leq k$.
(Any fixed points that arise are considered to be weak excedances if and only if they 
are in positions labeled by $J$.)
\end{definition}

A \emph{Grassmannian permutation} is a permutation 
$\pi = (\pi(1),\dots, \pi(n))$ which has at most one descent, 
i.e. at most one position $i$ such that $\pi(i) > \pi({i+1})$.
Note that the permutations of the form $\pi(J)$ defined above are precisely
the inverses of the Grassmannian permutations, since 
$\pi(J)^{-1} = (h_1, h_2, \dots, h_{n-k}, j_1, j_2, \dots, j_k)$.

In the case that $J=\{1,2,\dots, k\}$, the corresponding
TP Schubert cell is $(Gr_{k,n})_{>0}$
and $\pi(J) = \pi_{k,n}$.

To describe the partial order on positroid cells, it is
convenient to represent each decorated permutation
by a related affine permutation, as was done in \cite{KLS}. 
Then, the circular Bruhat order for $(Gr_{k,n})_{\geq 0}$
can be viewed as the restriction of the 
affine Bruhat order to the corresponding affine permutations,
together with a new bottom element $\hat{0}$ added to the poset.

\begin{definition}\label{def:affine}
Given a decorated permutation $\pi$ on $n$ letters, we construct its 
affinization $\tilde{\pi}$ as follows.
If $\pi(i)<i$, set $\tilde{\pi}(i):=\pi(i)+n$.
If $\pi(i)=i$ and $i$ is a clockwise fixed point, set 
$\tilde{\pi}(i):=i$.
If $\pi(i)=i$ and $i$ is a counterclockwise fixed point, set
$\tilde{\pi}(i):=i+n$.
Finally, if $\pi(i)>i$, set $\tilde{\pi}(i):=\pi(i)$.
\end{definition}

Clearly the affine permutation constructed above is a map
from $\{1,2,\dots,n\}$ to $\{1, 2,\dots, 2n\}$ such that
$i \leq \tilde{\pi}(i) \leq i+n$.
And modulo $n$, it agreeds with the underlying permutation $\pi$.

We now discuss plabic graphs, which e.g. are useful for constructing
parameterizations of positroid cells \cite{Postnikov}, and for 
understanding the cluster structure on the Grassmannian \cite{Scott}.

\begin{definition}
A \emph{planar bicolored graph} (or \emph{plabic graph}) is an undirected graph $G$ drawn inside a disk
(and considered modulo homotopy)
plus $n$ {\it boundary vertices\/} on the boundary of the disk,
labeled $1,\dots,n$ in clockwise order. 
The remaining {\it internal vertices\/}
are strictly inside the disk and are
colored in black and white.  We require that each boundary vertex is incident to a 
single edge.
\end{definition}

\begin{figure}[h]
\centering
\includegraphics[height=1in]{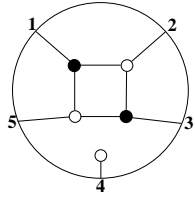}
\caption{A plabic graph}
\label{plabic}
\end{figure}

We will always assume that a plabic graph is {\it leafless}, i.e. that
it has no non-boundary leaves, and that it has no isolated components.
The following map gives a connection between plabic graphs and decorated permutations.

\begin{definition}\label{gen:trip}
Given a plabic graph $G$,
the \emph{trip} $T_i$ is the directed path which starts at the boundary vertex
$i$, and follows the ``rules of the road": it turns right at a
black vertex and   left at a white vertex.
Note that $T_i$  will also
end at a boundary vertex.  The \emph{decorated trip permutation} $\pi_G$
is the permutation such that $\pi_G(i)=j$ whenever $T_i$
ends at $j$.  Moreover, if there is a white (respectively, black)
boundary leaf at boundary vertex $i$, then $\pi_G(i)=i$ is a 
counterclockwise (respectively, clockwise) fixed point.
\end{definition}

We now 
define some local transformations of plabic graphs, which are analogous in some ways
to braid and commuting moves for reduced expressions in the symmetric group.

(M1) SQUARE MOVE.  If a plabic graph has a square formed by
four trivalent vertices whose colors alternate,
then we can switch the
colors of these four vertices.
\begin{figure}[h]
\centering
\includegraphics[height=.6in]{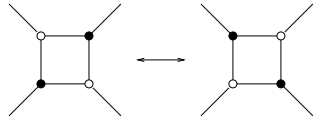}
\caption{Square move}
\label{M1}
\end{figure}

(M2) UNICOLORED EDGE CONTRACTION/UNCONTRACTION.
If a plabic graph contains an edge with two vertices of the same color,
then we can contract this edge into a single vertex with the same color.
We can also uncontract a vertex into an edge with vertices of the same
color.
\begin{figure}[h]
\centering
\includegraphics[height=.3in]{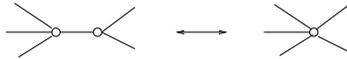}
\caption{Unicolored edge contraction}
\label{M2}
\end{figure}

(M3) MIDDLE VERTEX INSERTION/REMOVAL.
If a plabic graph contains a vertex of degree 2,
then we can remove this vertex and glue the incident
edges together; on the other hand, we can always
insert a vertex (of any color) in the middle of any edge.

\begin{figure}[h]
\centering
\includegraphics[height=.07in]{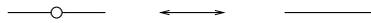}
\caption{Middle vertex insertion/ removal}
\label{M3}
\end{figure}

(R1) PARALLEL EDGE REDUCTION.  If a network contains
two trivalent vertices of different colors connected
by a pair of parallel edges, then we can remove these
vertices and edges, and glue the remaining pair of edges together.

\begin{figure}[h]
\centering
\includegraphics[height=.25in]{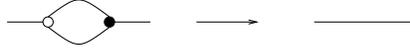}
\caption{Parallel edge reduction}
\label{R1}
\end{figure}

\begin{definition}\cite{Postnikov}
Two plabic graphs are called \emph{move-equivalent} if they can be obtained
from each other by moves (M1)-(M3).  The \emph{move-equivalence class}
of a given plabic graph $G$ is the set of all plabic graphs which are move-equivalent
to $G$.
A leafless plabic graph without isolated components
is called \emph{reduced} if there is no graph in its move-equivalence
class to which we can apply (R1).
\end{definition}

The following result is analogous to the Tits Lemma for reduced decompositions
of permutations.

\begin{theorem}\cite[Theorem 13.4]{Postnikov}
Two reduced plabic graphs 
are move-equivalent if and only if they have the same decorated
trip permutation.
\end{theorem}

A priori, it is not so easy to detect whether a given plabic graph is reduced.
One characterization was given in \cite[Theorem 13.2]{Postnikov}.
Another very
simple characterization 
was given in \cite{KW}.

\begin{definition}\label{labels}
Given a plabic graph $G$ with $n$ boundary vertices,
start at each boundary
vertex $i$ and label every edge along trip $T_i$ with $i$.
After doing this for each boundary vertex,
each edge will be labeled by up to two numbers (between $1$ and $n$).
If an edge is labeled by two numbers $i<j$,  write $[i,j]$
on that edge.

We say that a  plabic graph has the
\emph{resonance property}, if after labeling edges as above,
the set $E$ of edges incident to a given vertex  has
the following property:
\begin{itemize}
\item  there exist numbers $i_1<i_2<\dots<i_m$ such that when
we read the labels of $E$,  we see the labels
$[i_1,i_2],[i_2,i_3],\dots,[i_{m-1},i_m],[i_1,i_m]$ appear in
clockwise order.
\end{itemize}
\end{definition}

This property and the following characterization
of reduced plabic graphs were given in \cite{KW}.

\begin{theorem}\cite[Theorem 10.5]{KW}
A plabic graph is reduced
if and only if it has the resonance property.
\end{theorem}

\subsection{BCFW bridge decompositions}

Given a reduced plabic graph, it is easy to construct its corresponding
trip permutation, as explained in Definition 
\ref{gen:trip}.  But how can we go backwards, i.e. given the permutation, how can we construct 
a corresponding reduced plabic graph?  One procedure to do this was given in 
\cite[Section 20]{Postnikov}.  Another elegant solution -- the 
\emph{BCFW-bridge construction} -- was given in 
\cite[Section 3.2]{Scatt}.

To explain the BCFW-bridge construction, we use
the representation of each decorated permutation $\pi$ as an affine permutation
$\tilde{\pi}$,
see Definition \ref{def:affine}.
We say that position $i$ of $\tilde{\pi}$ is a \emph{fixed point}
if $\tilde{\pi}(i) = i$ or 
$\tilde{\pi}(i) = i+n$.  And we say that 
$\tilde{\pi}$ is a \emph{decoration of the identity} if it has a fixed point
in each position.

\begin{definition}[\emph{The BCFW-bridge construction}]\label{def:BCFW}
Given a decorated permutation $\pi$ on $n$ letters, which is not 
simply a decoration of the identity, we start by  choosing
a pair of adjacent positions $i<j$ such that 
$\tilde{\pi}(i)<
\tilde{\pi}(j)$, and 
$i$ and $j$ are not fixed points.
Here \emph{adjacent} means that either
$j=i+1$, or $i<j$ and every position $h$ such that $i<h<j$ is a fixed point.
We record the transposition $\tau= (ij)$ and swap the entries in positions
$i$ and $j$ in $\tilde{\pi}$.  Any entries in the resulting permutation
which are fixed points are designated as \emph{frozen}, and henceforth ignored.
We continue this process on the resulting affine permutation,
until the end result is a decoration of the identity.
Finally we use the sequence of transpositions $\tau$ as a recipe
for adding ``bridges," thereby constructing a plabic graph.
\end{definition}

See Table \ref{table:bridges} and Figure \ref{fig:bridges}
for an example of a bridge decomposition of the 
permutation $\pi=(4,6,5,1,2,3)$ (with corresponding
affine permutation $\tilde{\pi} = (4,6,5,7,8,9)$).
Here a \emph{bridge} is the subgraph shown
at the left of Figure \ref{fig:bridges}, and a bridge decomposition
is a graph built by attaching successive bridges.
The sequence of transpositions
$\tau$ gives a recipe for where to place successive bridges.

\begin{center}
        \begin{table}[h]
                \begin{tabular}{| l | l |}
                \hline
                 & \ 1 \ 2 \ 3 \ 4 \ 5 \ 6 \\
                \ $\tau$ & \ $\downarrow$ \ $\downarrow$ \ $\downarrow$ \ $\downarrow$ \ $\downarrow$ \ $\downarrow$ \\
                \hline
               $(34)$  & \ 4 \ 6 \ 5 \ 7 \ 8 \ 9 \\
               $(23)$ & \ 4  \ 6 \  7 \ 5 \ 8 \ 9 \\ 
               $(12)$ & \ 4 \ 7 \ 6 \ 5 \ 8 \ 9 \\
               $(56)$ & \framebox{7}  4 \ 6 \ 5 \ 8 \ 9 \\
               $(45)$ & \framebox{7}  4 \ 6 \ 5 \ 9 \ 8 \\
               $(34)$ & \framebox{7}  4 \ 6 \ 9 \framebox{5}  8 \\
               $(46)$ & \framebox{7}  4 \framebox{9}  6 \framebox{5}  8 \\
               $(24)$ & \framebox{7}  4 \framebox{9}  8 \framebox{5}\framebox{6} \\
                \hline
                      & \ 7 \ 8 \ 9 \ 4 \ 5 \ 6 \\
                \hline
                \end{tabular}
\vspace{.5cm}
       \caption{A BCFW-bridge decomposition of $\tilde{\pi} = (4,6,5,7,8,9)$.
The frozen entries are boxed.}
       \label{table:bridges}
        \end{table}
\end{center}

\begin{proposition}\cite[Section 3.2]{Scatt}
Given a decorated permutation $\pi$,
the BCFW-bridge construction constructs a reduced plabic graph
whose trip permutation is $\pi$.
\end{proposition}

\begin{figure}[h]
\centering
\hspace{1cm}
\includegraphics[height=1in]{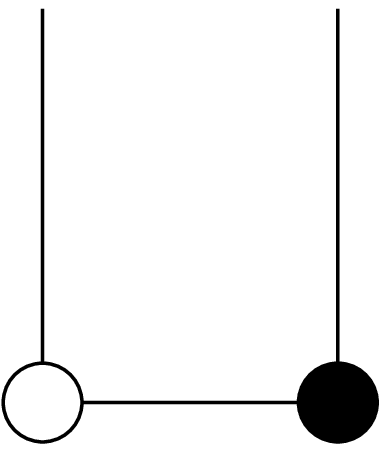}\hspace{1in}
\includegraphics[height=1.5in]{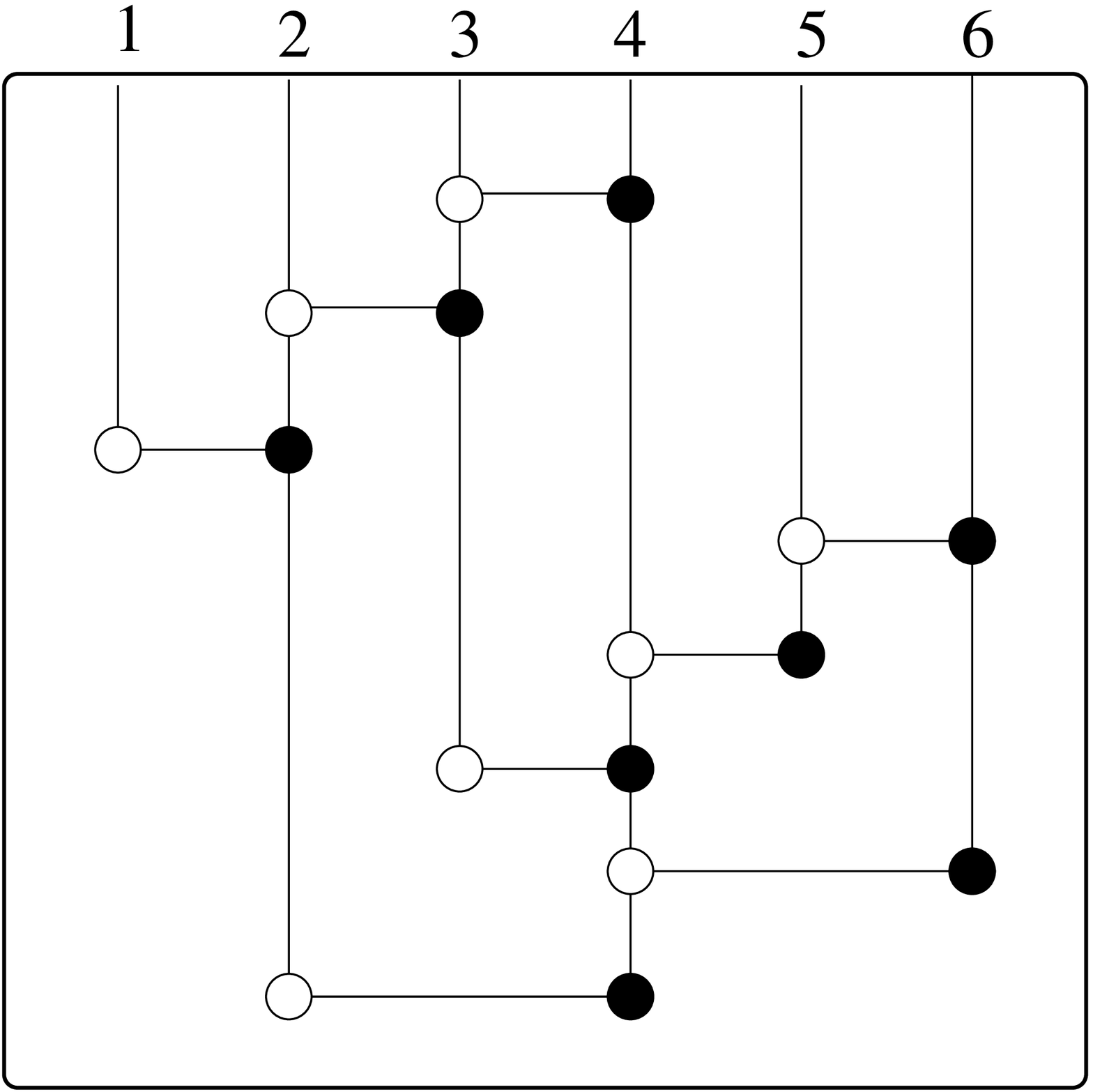}
\caption{A single bridge, and the plabic graph $G$ associated to the bridge decomposition from Table \ref{table:bridges}.  Note that the trip permutation of $G$
equals the product $(24)(46)(34)(45)(56)(12)(23)(34)$ of transpositions
$\tau$ in the bridge decomposition, which equals 
the corresponding decorated permutation $\pi = (4,6,5,1,2,3)$.}
\label{fig:bridges}
\end{figure}

\section{Bridge polytopes and their one-skeleta}

The goal of this section is to introduce some 
polytopes that we will call  
\emph{bridge polytopes}, because their one-skeleta encode 
BCFW-bridge decompositions of reduced plabic graphs.  This statement is 
analogous to the fact that 
the one-skeleton of the permutohedron encodes
reduced decompositions of $w_0$.  More specifically,
for each $k$-element subset $J$ of $[n]$, we will introduce
a bridge polytope $\Br_J$ which encodes bridge decompositions
of reduced plabic graphs for the permutation $\pi(J)$
labeling the corresponding TP Schubert cell.
In the case that $J = \{1,2,\dots,k\}$,
we will also denote $\Br_J$ by $\Br_{k,n}$ -- this will be the polytope 
encoding bridge decompositions of $\pi_{k,n}$,
the decorated permutation labeling the totally positive Grassmannian.

\begin{figure}[h]
\centering
\includegraphics[height=5.3cm]{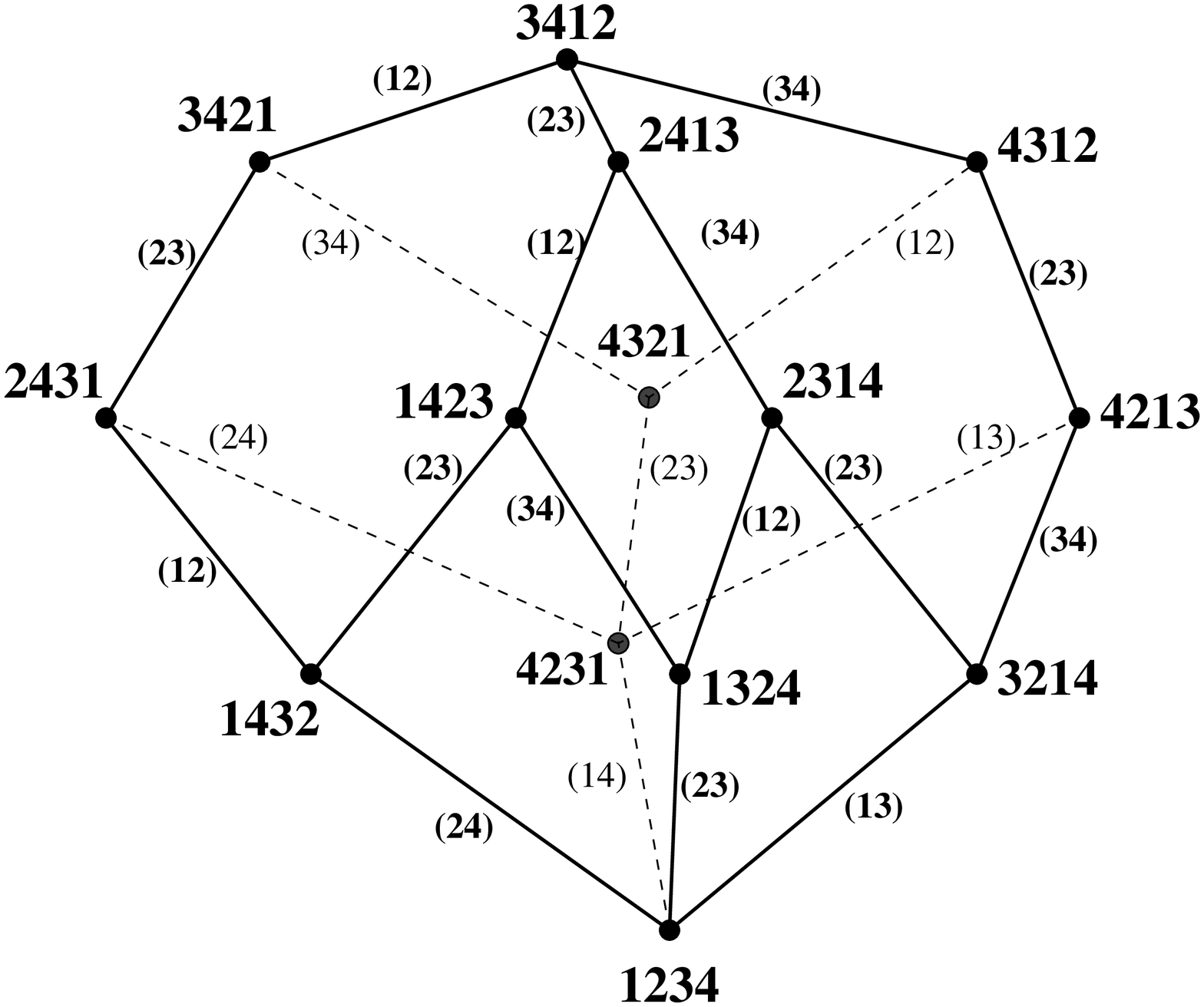} \hspace{1cm}
\includegraphics[height=5.3cm]{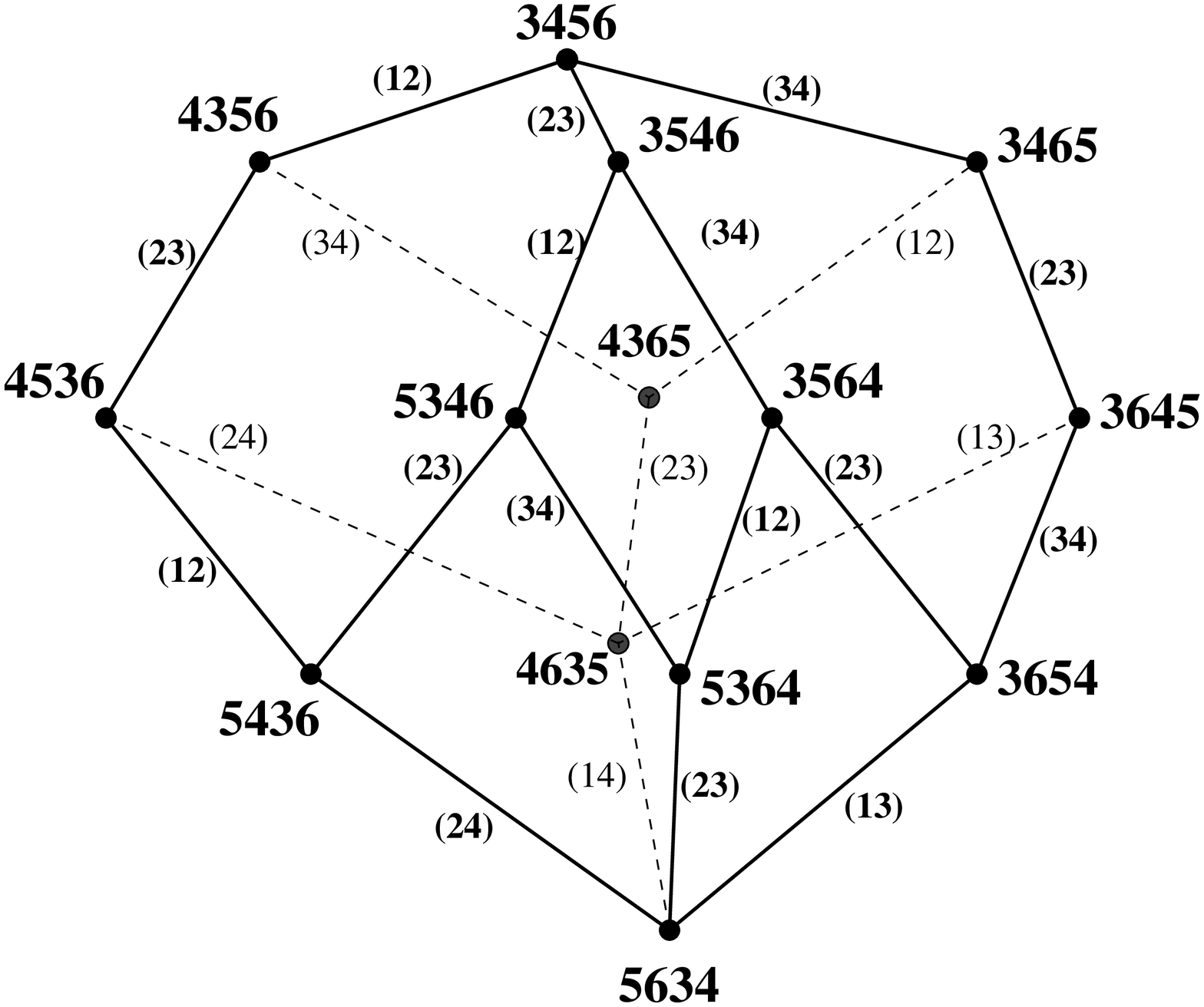}
\caption{Two labelings of a bridge polytope.  At the left,
vertices are labeled by ordinary permutations, and edges are labeled by 
the pair of \emph{values} which are swapped.  At the right,
vertices are labeled by affine permutations, and edges are labeled by 
the pair of \emph{positions} which are swapped.  
The map from vertex-labels at the left to vertex-labels at the right is 
$z \mapsto \widetilde{z^{-1}}$.
In both cases, the paths along the one-skeleton 
from top to bottom encode bridge decompositions of the permutation
$\pi_{2,4} = (3,4,1,2)$.}
\label{fig:scamp}
\end{figure}

Before giving the general construction, we present an example in Figure 
\ref{fig:scamp}, which encodes the bridge decompositions of plabic graphs with
trip permutation 
$\pi = (3,4,1,2)$.  The figure at the left shows the polytope 
obtained by taking the convex hull of the permutations
$$\{(z_1,z_2,z_3,z_4) \ \vert \ z_1 \geq 1, z_2 \geq 2, z_3 \leq 3, z_4 \leq 4)\}.$$
The edge between two permutations is labeled by the 
\emph{values} of the swapped entries.
The figure at the right shows the same polytope, but now vertices are labeled
by affine permutations.  A vertex which was labeled by $z$ at the left
is labeled at the right by the affinization $\widetilde{z^{-1}}$ of $z^{-1}$.
Note that we took the inverse in order to get a vertex-labeling of 
the polytope such that edges correspond to the \emph{positions}
of the swapped entries, agreeing with the sequence of 
transpositions defined in Definition \ref{def:BCFW}.

The  main result of this section will imply
that the minimal chains in the one-skeleton 
of the bridge polytope from ``top" to ``bottom"
(from $(3,4,1,2)$ to $(1,2,3,4)$ using the vertex-labeling at the left,
and from $(3,4,5,6)$ to $(5,6,3,4)$ using the vertex-labeling at the right)
are in bijection with the bridge decompositions of the reduced plabic graphs
with trip permutation $\pi_{2,4} = (3,4,1,2)$.
For example,  the leftmost chain from top to bottom is labeled by 
the sequence of transpositions $(12)$, $(23)$, $(12)$, $(24)$;  the 
corresponding plabic graph is shown in Figure \ref{fig:bridge2}.

\begin{figure}[h]
\centering
\includegraphics[height=3.3cm]{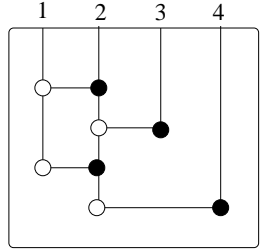} 
\caption{The plabic graph coming from the sequence of bridges
$(12)$, $(23)$, $(12)$, $(24)$.  Note that its trip permutation 
is $(3,4,1,2)=(24)(12)(23)(12)$.}
\label{fig:bridge2}
\end{figure}


We now turn to the general construction.

\begin{definition}\label{def:bridge}
Let $J \subset \{1,2,\dots, n\}$ and set
\begin{equation}
S_J = \{\pi \in S_n \ \vert \ \pi(j) \geq j \text{ for }j\in J
\text{ and }\pi(j) \leq j \text{ for }j \notin J.\}
\end{equation}
In other words, any permutation in $S_J$ is required to have a
\emph{weak excedance} in position $j$ for each $j\in J$, and a 
\emph{weak non-excedance} in position $j$ for each $j\notin J$.

We define a polytope $\G_J$ by
\begin{equation}
\G_J = \conv\{ (\pi(1),\pi(2),\dots, \pi(n)) \ \vert \ \pi\in S_J\}
\subset \R^n.
\end{equation}
\end{definition}

In the special case that $J = \{1,2,\dots, k\}$, we write
$$\G_{k,n} = \G_J = \{\pi \in S_n \ \vert \ 
\pi(j) \geq j \text{ for } 1 \leq j \leq k \text{ and }
\pi(j) \leq j \text{ for } k+1 \leq j \leq n.\}$$

Recall the definition of $\pi(J)$ from Definition 
\ref{def:TPSchubert}, the decorated permutation labeling the 
corresponding TP Schubert cell.  Let $e(J)$ be the decoration
of the identity which has counterclockwise fixed points 
precisely in positions $J$.

The following is our main result.  We describe it using 
the vertex-labeling of the bridge polytope with ordinary 
permutations, as in Definition \ref{def:bridge},
but it can be easily translated using 
the vertex-labeling by affine permutations.

\begin{theorem}\label{thm:main}
Let $J$ be a $k$-element subset of $[n]$.
The shortest paths from $\pi(J)$ to the identity permutation $e$
along the one-skeleton of the bridge polytope $\G_J$ 
are in bijection with the BCFW-bridge decompositions of 
the permutation $\pi(J)^{-1}$, where a sequence of edge-labels
in a path is interpreted as a sequence of transpositions $\tau$ in the 
bridge decomposition.  Equivalently,
there is an edge between two vertices $\pi$ and $\hat{\pi}$ of 
the bridge polytope $\G_J$ if and only if there exists some
pair  $i<\ell$ such that
$(i \ell) \pi = \hat{\pi}$, and 
$\pi(j) = \hat{\pi}(j) = j$ for $i<j<\ell$.
In other words, $\pi$ and $\hat{\pi}$ differ by 
swapping the values $i$
and $\ell$, 
and $i+1, i+2, \dots, \ell-1$ are fixed points
of both $\pi$ and $\hat{\pi}$.
\end{theorem}

We will prove Theorem \ref{thm:main} in Section \ref{sec:Bruhat}.

The following corollary is immediate from Theorem \ref{thm:main}.

\begin{corollary}
The shortest paths
in the bridge polytope $\G_{k,n}$ from $\pi_{k,n}$ down to the identity $e$ are in bijection 
with the BCFW-bridge decompositions of the permutation 
$\pi_{n-k,n}$.  In other words, we can read off all 
BCFW-bridge decompositions of $\pi_{n-k,n}$ by 
recording the edge-labels on all shortest paths from $\pi_{k,n}$
to $e$ in the one-skeleton of 
$\G_{k,n}$.
\end{corollary}

\section{Bruhat interval polytopes and the proof of Theorem 
\ref{thm:main}}\label{sec:Bruhat}

Bruhat interval polytopes are a class of polytopes which were recently
studied by Kodama and the second author in \cite{KW3}, 
in connection with the full Kostant-Toda lattice on the 
flag variety.  The combinatorial properties of Bruhat interval
polytopes were further investigated by Tsukerman and the second author
in \cite{TW}.  We will show that bridge polytopes 
are isomorphic to certain Bruhat interval polytopes,
and use a result from \cite{KW3} as a tool for proving 
Theorem \ref{thm:main}.  We will also deduce 
a characterization of the one-skeleta of a large class of 
Bruhat interval polytopes.

\begin{definition}\label{def:BIP}
Let $u,v \in S_n$ such that $ u \leq v $ in (strong) Bruhat
order.  
The \emph{Bruhat interval polytope}
$\mathsf{Q}_{u,v}$ is defined as the convex hull
$$\mathsf{Q}_{u,v}
= \conv\{(z(1),\dots,z(n)) \ \vert \ \text{ for }z\in S_n \text{ such that }
u \leq z \leq v\}.$$
\end{definition}

Note that when $u$ is the identity and $v = w_0$,
the longest permutation in $S_n$, the Bruhat interval polytope
$\mathsf{Q}_{u,v}$ equals the permutohedron $\Perm_n$.

\begin{lemma}\label{lem:bridge-Bruhat}
The bridge polytope $\G_J$ is isomorphic to the Bruhat interval polytope
$\mathsf{Q}_{e,\pi(J)^{-1}}$.  
More specifically, 
recall that for $J$ a $k$-element subset of $[n]$, we write 
$J = \{j_1 < \dots < j_k\}$, and $J^c = \{h_1 < \dots < h_{n-k}\}$.
Consider the map $\psi:\R^n \to \R^n$ defined by 
$\psi(x_1,\dots,x_n) = (x_{h_1}, x_{h_2}, \dots, 
x_{h_{n-k}}, x_{j_1}, x_{j_2}, \dots, x_{j_k})$.
Then $\Psi$ is an isomorphism from 
$\G_J$ to $\mathsf{Q}_{e,\pi(J)^{-1}}$.
\end{lemma}

\begin{proof}
Since $\psi$ simply permutes the coordinates of $\R^n$,
it obviously acts as an isomorphism on any polytope.  
Moreover, $\psi$ takes the vertex $(z(1),\dots, z(n))$
labeled by 
the permutation $z$ in $\G_{k,n}$ to the vertex 
labeled by the permutation $z \pi(J)^{-1}$ of 
$\mathsf{Q}_{e,\pi(J)^{-1}}$.  In particular, 
it takes the vertices $\pi(J)$ and $e$ of $\G_{k,n}$
to the vertices $e$ and $\pi(J)^{-1}$ of $\mathsf{Q}_{e,\pi(J)^{-1}}$,
respectively.
It is not hard to show that $\psi$ maps 
the set of permutations $S_J$ to the set of permutations
in the interval $[e,\pi(J)^{-1}]$.
\end{proof}

Once we have established Theorem \ref{thm:main}, 
Lemma \ref{lem:bridge-Bruhat} immediately implies the following description 
of the  one-skeleton of any
Bruhat interval polytopes $\mathsf{Q}_{e,w}$,
when $w$ is a Grassmannian permutation.  

\begin{corollary}
Let $w$ be a Grassmannian permutation in $S_n$, i.e. a permutation
with at most one descent.  Then there is an edge between 
vertices $u$ and $v$ in $\mathsf{Q}_{e,w}$ if and only if 
$u$ and $v$ satisfy the following properties:
\begin{itemize}
\item $v = (i \ell) u$
\item in the permutations $u$ and $v$,
each of the values $i+1, i+2,\dots, \ell-1$ are in precisely 
the same positions that they are in $w$.
\end{itemize}
\end{corollary}

We now prove Theorem \ref{thm:main}, by establishing 
Lemma \ref{lem:edge-swap}, 
Proposition \ref{prop:fixed},
and Theorem \ref{th:edgeexist}, below.

\begin{lemma}\label{lem:edge-swap}
If there is an edge in $\G_J$ between $\pi$ and $\hat{\pi}$, then 
there exists some transposition $(i \ell)$ such that 
$(i \ell) \pi = \hat{\pi}$.
\end{lemma}
\begin{proof}
By \cite[Theorem A.10]{KW3}, every 
edge of a Bruhat interval polytope connects two vertices
$u$ and $v$, where $v = (i \ell)$ for some transposition $(i \ell)$.
Lemma \ref{lem:edge-swap} now follows from 
Lemma \ref{lem:bridge-Bruhat}.
\end{proof}

\begin{remark}
It was shown more generally in \cite{TW} that every face of a Bruhat interval 
polytope is a Bruhat interval polytope.
\end{remark}


\begin{proposition}\label{prop:fixed}
Suppose there's an edge in $\G_J$ between $\pi$ and $\hat{\pi}$ where
$(i \ell) \pi = \hat{\pi}$ for some $i<\ell$, i.e. $\pi$ and $\hat{\pi}$
differ by swapping the values $i$ and $\ell$.  Then $i+1, i+2, \dots, \ell-1$
must be fixed points of $\pi$ and $\hat{\pi}$.
\end{proposition}

\begin{proof}
We use a proof by contradiction.  Suppose that not all of 
$i+1, i+2,\dots, \ell-1$ are fixed points. Let $j$ be the smallest
element of $\{i+1,i+2,\dots, \ell-1\}$ which is not a fixed point. 
Let $a, b$, and $c$ be the positions of $i, \ell$, and $j$ in $\pi$.

Without loss of generality, $\pi_c = j > c$.  So $c\in J$
and is a position of a weak excedance in any permutation in $S_J$.
Since we have an edge in the polytope between $\pi$
and $\hat{\pi}$, there exists a $\lambda\in \R^n$ 
such that $\lambda \cdot x: = \sum_h \lambda_h x_h$ (applied to $x\in S_J$)
is maximized  precisely on $\pi$ and $\hat{\pi}$.  In particular,
we must have $\lambda_a = \lambda_b$.

If $\lambda_a = \lambda_b < \lambda_c$, then define
$\tilde{\pi}$ so that it agrees with $\pi$ except in positions $b$
and $c$, where we have $\tilde{\pi}_b = j$ and $\tilde{\pi}_c = \ell$.
Since $\pi_b = \ell$ and $\hat{\pi}_b = i$ and $i<j<\ell$,
permutations in $S_J$ are allowed to have a $b$th coordinate of $j$.
Since $j>c$ and $j<\ell$, we have $\ell>c$.  So permutations in $S_J$
are allowed to have a $c$th coordinate of $\ell$. Therefore
$\tilde{\pi} \in S_J$.  But $\lambda \cdot \tilde{\pi} > 
\lambda \cdot \pi$, which is a contradiction.

If $\lambda_a = \lambda_b > \lambda_c$, then define
$\tilde{\pi}$ so that it agrees with $\pi$ except in positions
$a$ and $c$, where we have $\tilde{\pi}_a = j$ and 
$\tilde{\pi}_c=i$.  Since $\pi_a=i$ and $\hat{\pi}_a=\ell$ and $i<j<\ell$,
permutations in $S_J$ are allowed to have an $a$th coordinate of $j$.
Since $\pi_c = j>c$, and $i+1, i+2,\dots, j-1$ are fixed points
of $\pi$, $c \notin \{i+1, i+2,\dots, j-1\}$.  So $c \leq i$, 
and permutations in $S_J$ are allowed to have a $c$th coordinate of $i$
(recall that $c\in J$ and hence $c$ must be a weak excedance 
in any element of $S_J$).  Therefore $\tilde{\pi} \in S_J$.  
But $\lambda \cdot \tilde{\pi} > \lambda \cdot \pi$, which is a contradiction.
\end{proof}

\begin{theorem}\label{th:edgeexist}
Consider the bridge polytope $\G_J$ and two permutations 
$\pi$ and $\hat{\pi}$ in $S_J \subset S_n$ such that 
$(i \ell) \pi = \hat{\pi}$ and $i+1, i+2, \dots, \ell-1$ are fixed points
of $\pi$ and $\hat{\pi}$.  Choose a vector $\lambda = 
(\lambda_1,\dots, \lambda_n)\in \R^n$ whose coordinates are some 
permutation of the numbers $1, n^2, n^4, \dots, n^{2(n-1)}$, 
and such that
\begin{equation*}
\lambda_{\pi^{-1}(1)} < \lambda_{\pi^{-1}(2)} < \dots < \lambda_{\pi^{-1}(i-1)} 
< 
\Box < \lambda_{\pi^{-1}(i)} = \lambda_{\pi^{-1}(\ell)} < \Box < 
\lambda_{\pi^{-1}(\ell+1)} < \dots < \lambda_{\pi^{-1}(n-1)} < 
\lambda_{\pi^{-1}(n)}.
\end{equation*}
Here the $\Box$ at the left represents the coordinates
$\lambda_{\pi^{-1}(j)}$ for $j\in J$ and $i<j<\ell$, sorted from 
left to right in increasing order of $j$, and the $\Box$ at the right
represents the coordinates
$\lambda_{\pi^{-1}(j)}$ for $j\notin J$ and $i<j<\ell$, sorted from 
left to right in increasing order of $j$.

Then when we calculate the dot product $\lambda \cdot x$
for each $x\in S_J$, $\lambda \cdot x$ is maximized precisely on 
$\pi$ and $\hat{\pi}$.  In particular, there is an edge in 
$\Br_J$ between $\pi$ and $\hat{\pi}$.
\end{theorem}

\begin{proof}
To prove Theorem \ref{th:edgeexist}, we will first use the fact that the coordinates of $\lambda$
are $1, n^2, \dots, n^{2(n-1)}$ to show that to solve for the 
$x\in S_J$ on which $\lambda \cdot x$ is maximized, it suffices
to use the greedy algorithm.  In other words, we will show that 
we can construct
such $x$ by maximizing individual coordinates $x_i$ one by one,
starting with the $x_i$ where $\lambda_i$ is maximal, and continuing
in decreasing order of the values of $\lambda_j$.
Next we will show that when we run the greedy algorithm, we will
get precisely the outcomes $\pi$ and $\hat{\pi}$.

Let us show now that in order to maximize the dot product
$\lambda \cdot x$ for $x \in S_J$, the greedy algorithm works.
Recall that the coordinates of $\lambda$ are the values
$1, n^2, \dots, n^{2(n-1)}$ (in some order).  Let $h \in [n]$ be 
such that $\lambda_h = n^{2(n-1)}$.  We first claim
that we need to maximize $x_h$ subject to the condition $x\in S_J$.
Let $w_h$ be that maximum possible value.
Let $y\in S_J$ be some other permutation with $y_h \leq w_h-1$.
Then $\lambda_h x_h - \lambda_h y_h \geq n^{2(n-1)}$.
And since the maximum difference of $x_j$ and $y_j$ is $n-1$,
we have
$$\sum_{1 \leq j \leq n, j \neq h} \lambda_j y_j - 
\sum_{1 \leq j \leq n, j \neq h} \lambda_j x_j 
\leq (n-1)(1+n^2+ \dots + n^{2(n-2)}).$$
Now note that  since $n-1<n$ and $1+n^2+ \dots + n^{2(n-2)}$ has $n-1$
terms, we have 
$(n-1)(1+n^2+ \dots + n^{2(n-2)}) < n \cdot n \cdot n^{2(n-2)} = n^{2(n-1)}$.
It follows that $\lambda \cdot x > \lambda \cdot y$.

Next let $h' \in [n]$ be such that $\lambda_{h'} = n^{2(n-2)}$.
Using the same argument, and the inequality 
$(n-1)(1+n^2+ \dots +n^{2(n-3)}) < n^{2(n-2)}$, it follows that to maximize
$\lambda \cdot x$, we need to now choose the maximum value of 
$x_{h'}$ subject to the condition $x \in S_J$.  Continuing in this fashion,
we see that this greedy algorithm will compute all $x\in S_J$ such 
that $\lambda \cdot x$ is maximized.

Now we will show that if we use the greedy algorithm, choosing coordinates
$x_j$ for $x\in S_J$ in decreasing order of the corresponding value 
$\lambda_j$, we will get precisely the permutations $\pi$ and $\hat{\pi}$.

{\bf Step 1.}  
Looking at the ordering of the coordinates of $\lambda$ as 
described in Theorem \ref{th:edgeexist}, we first need to maximize
$x_{\pi^{-1}(n)}$, then $x_{\pi^{-1}(n-1)}, \dots,$ then 
$x_{\pi^{-1}(\ell+1)}$.  Therefore we place $n$ in position 
$\pi^{-1}(n)$, $n-1$ in position $\pi^{-1}(n-1)$, $\dots$,
$\ell+1$ in position $\pi^{-1}(\ell+1)$.  (Note that there exist 
permutations $x\in S_J$ with 
these coordinates in these positions -- for example, $\pi$ and $\hat{\pi}$ 
are two examples of such permutations.)

{\bf Step 2.}  Next we need to maximize the values that we put in 
positions $\pi^{-1}(j)$, for 
$i<j<\ell$ and 
$\pi^{-1}(j) = j \notin J$. 
But these are positions of weak non-excedances in $S_J$, so the
best we can do is to put fixed points there.  Note that $\pi$
and $\hat{\pi}$ also have 
fixed points in these positions, so the $x$ we are building so far
agrees with both $\pi$ and $\hat{\pi}$.

{\bf Step 3.}  Now we want to maximize the values that we can 
put in positions $\pi^{-1}(i)$ and $\pi^{-1}(\ell)$.  The greatest
unused value so far is $\ell$ so we can put that in either position
$\pi^{-1}(\ell)$ (agreeing with $\pi$) or $\pi^{-1}(i)$ (agreeing
with $\hat{\pi}$).

Now let $j$ be the maximum value that we have not yet placed in 
any position of the $x$ we are building.  If $i<j<\ell$, then 
necessarily $j\in J$, i.e. $j$ is a position where all permutations
of $S_J$  must have
a weak excedance.  So we must make $j$ a fixed point in $x$ (the only
other option is to place a value smaller than $j$ in position $j$).  
Similarly for any other $i<j<\ell$, we must set $x_j = j$.

Having done this, the maximum value that we have not yet placed in $x$
is $i$.  So we now put $i$ in position $\pi^{-1}(i)$ (or 
position $\pi^{-1}(\ell)$).  Note that the $x$ we are building agrees
with either $\pi$ or $\hat{\pi}$ so far.

{\bf Step 4.}  Finally we want to maximize the values that we place
in positions $\pi^{-1}(i-1),\dots, \pi^{-1}(1)$.  The remaining
unused values are $i-1,\dots, 1$.  So for $i-1 \geq j \geq 1$,
we place $j$ in position
$\pi^{-1}(j)$.  

This greedy algorithm has 
constructed precisely two permutations, $\pi$ and $\hat{\pi}$, and clearly 
$\lambda \cdot \pi = \lambda \cdot \hat{\pi}$.
Therefore if we consider the values $\lambda \cdot x$ for 
$x\in S_J$, $\lambda \cdot x$ is maximized precisely on $\pi$ and 
$\hat{\pi}$.  It follows that there is an edge in $\Br_J$ between
$\pi$ and $\hat{\pi}$.
\end{proof}

\section{The two-dimensional faces of bridge polytopes}

In this section we describe the two-dimensional faces of bridge polytopes,
and explain how they are related to the moves for reduced plabic graphs.

\begin{theorem}
A two-dimensional face of a bridge polytope is either a square, a 
trapezoid, or a regular hexagon, with labels as in Figure \ref{fig:faces}.
\end{theorem}

\begin{figure}[h]
\centering
\includegraphics[height=1.2in]{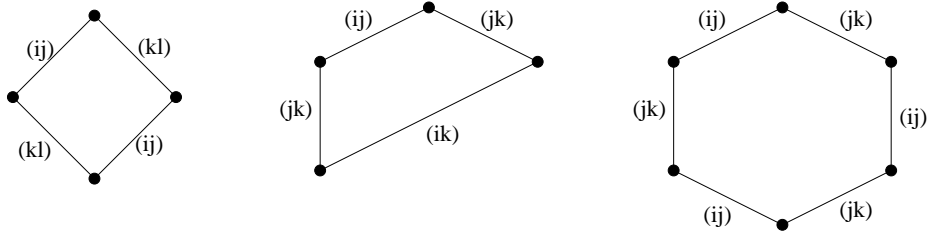}
\caption{The possible two-dimensional faces of a bridge polytope.  
Here we have
$i<j<k<l$ for the square,
$i<j<k$ or $k<j<i$ for the trapezoid, and $i<j<k$ for the hexagon.} \label{fig:faces}
\end{figure}

\begin{proof}
Consider a two-dimensional face of a bridge polytope.  By Lemma \ref{lem:edge-swap},
each edge is parallel to a vector of the form $e_i - e_{\ell}$, and has length
an integer multiple of $\sqrt{2}$.  
A simple calculation with dot products show that the 
possible angles among the edges are $\frac{\pi}{3}, \frac{\pi}{2}, \frac{2\pi}{3}$:
more precisely,
two vectors $e_i-e_j$ and $e_k-e_l$ (where $i, j, k, l$ are distinct) have angle $\frac{\pi}{2}$;
two vectors $e_i-e_j$ and $e_i-e_k$ have angle $\frac{\pi}{3}$;
and two vectors $e_i-e_j$ and $-e_i+e_k$ have angle $\frac{2\pi}{3}$.

The only possibilities for such a polygon are:
a square,
a trapezoid (with angles $\frac{\pi}{3}, \frac{\pi}{3}, \frac{2\pi}{3}, \frac{2\pi}{3}$),
a regular hexagon (all angles are $\frac{2\pi}{3}$),
an equilateral triangle, or a parallelogram.
In the first three cases, the labels on the edges of the polygons must be as 
in Figure \ref{fig:faces}. In the case of the square, it follows from 
the rules of bridge decompositions that the intervals $[i,j]$ and $[k,l]$
must be disjoint, and hence without loss of generality $i<j<k<l$.
(E.g. $k$ cannot lie in between $i$ and $j$, because 
in order to perform the swap $i$ and $j$, all elements between $i$ and $j$ must
be fixed points.
And we are not allowed to swap fixed points.)
A similar argument explains the ordering on $i,j,k$ for the trapezoid and the
hexagon. 

We now argue that it is impossible for a two-dimensional face to be a triangle
or a parallelogram.
Note that if one traverses any cycle in the one-skeleton of a bridge polytope, 
the product of the corresponding edge labels must be $1$.
It follows that there cannot be a two-face with an odd number of sides, because the product of an 
odd number of transpositions is an odd permutation, and hence is never the identity.
Therefore an equilateral triangle is impossible.
Finally consider the case of a face which is a parallelogram.  Its edge labels must 
have alternating labels $(ij)$, $(ik)$, $(ij)$, $(ik)$.  But 
the product $(ij) (ik) (ij) (ik)$ is not equal to the identity permutation.
\end{proof}

\begin{remark}
The same proof shows that the face of any Bruhat interval polytope is a square,
a trapezoid, or a hexagon, as shown
in Figure \ref{fig:faces}.
\end{remark}

\begin{theorem}\label{th:faces-moves}
The three kinds of two-dimensional faces of bridge polytopes correspond to
simple applications of the local moves
for plabic graphs.  More specifically, consider 
a shortest path $p$ from $\pi_{k,n}$ (or more generally $\pi(J)$)
to $e$ in the one-skeleton  of the bridge polytope $\Br_{k,n}$
(or more generally $\Br_J$). Choose a two-dimensional face $F$ such that $p$ traverses
half the sides of $F$, and modify $p$ along $F$, obtaining a new path $p'$ which 
goes around the other sides of $F$.
Then the reduced plabic graphs corresponding to $p$ and $p'$ are related by 
homotopy and by the local moves for plabic graphs as in Figure \ref{fig:faces-moves}.
\end{theorem}

\begin{proof}
The proof of Theorem \ref{th:faces-moves} is illustrated in Figure \ref{fig:faces-moves}.
The two reduced plabic graphs related by a square face in a bridge polytope 
are related by homotopy.
The two plabic graphs related by a trapezoidal face are related by moves of type
(M3).  And the two plabic graphs related by a hexagonal face are related by a combination
of moves (M1) and (M3).  Note that in the latter case, the dashed edge in Figure 
\ref{fig:faces-moves} must be present, or else the plabic graph associated to our bridge
decomposition would not be reduced.
\begin{figure}[h]
\centering
\includegraphics[height=3.5in]{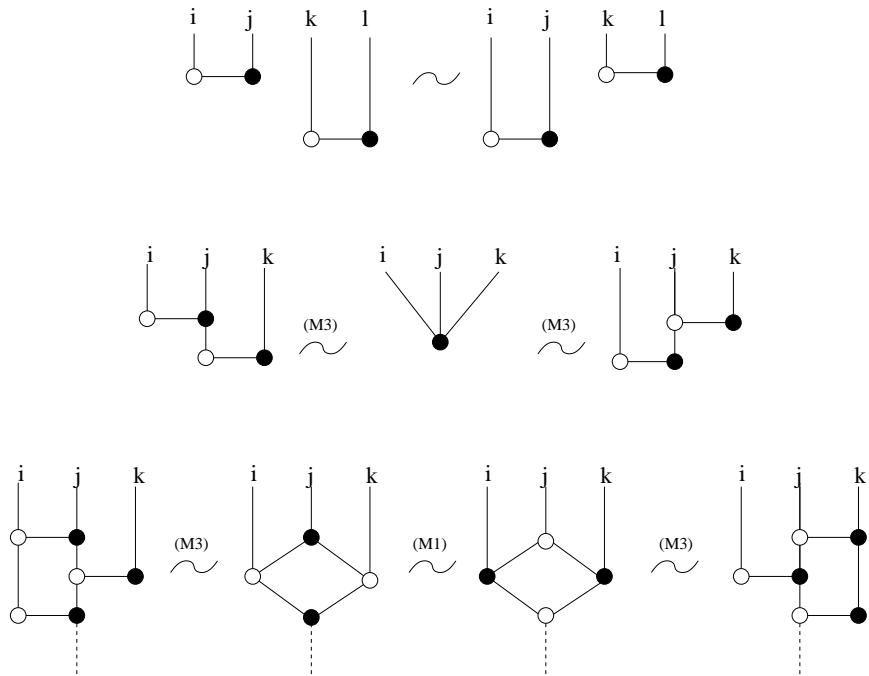}
\caption{A graphical proof of Theorem \ref{th:faces-moves}.}
\label{fig:faces-moves}
\end{figure}
\end{proof}





\bibliographystyle{alpha}
\bibliography{bibliography}

\newcommand{\etalchar}[1]{$^{#1}$}
\def\cprime{$'$} \def\cprime{$'$} \def\cprime{$'$} \def\cprime{$'$}
\begin{thebibliography}{AHBC{\etalchar{+}}12}

\bibitem[AHBC{\etalchar{+}}12]{Scatt}
Nima Arkani-Hamed, Jacob~L. Bourjaily, Freddy Cachazo, Alexander~B. Goncharov,
  Alexander Postnikov, and Jaroslav Trnka.
\newblock Scattering amplitudes and the positive {G}rassmannian, 2012.
\newblock Preprint, \texttt{arXiv:1212.5605}.

\bibitem[BLVS{\etalchar{+}}99]{OM}
Anders Bj{\"o}rner, Michel Las~Vergnas, Bernd Sturmfels, Neil White, and
  G{\"u}nter~M. Ziegler.
\newblock {\em Oriented matroids}, volume~46 of {\em Encyclopedia of
  Mathematics and its Applications}.
\newblock Cambridge University Press, Cambridge, second edition, 1999.

\bibitem[BW82]{BW}
Anders Bj{\"o}rner and Michelle Wachs.
\newblock Bruhat order of {C}oxeter groups and shellability.
\newblock {\em Adv. in Math.}, 43(1):87--100, 1982.

\bibitem[Ede81]{Edelman}
Paul~H. Edelman.
\newblock The {B}ruhat order of the symmetric group is lexicographically
  shellable.
\newblock {\em Proc. Amer. Math. Soc.}, 82(3):355--358, 1981.

\bibitem[KLS13]{KLS}
Allen Knutson, Thomas Lam, and David~E. Speyer.
\newblock Positroid varieties: juggling and geometry.
\newblock {\em Compos. Math.}, 149(10):1710--1752, 2013.

\bibitem[KW11]{KW}
Y.~Kodama and L.~K. Williams.
\newblock K{P} solitons, total positivity, and cluster algebras.
\newblock {\em Proc. Natl. Acad. Sci. USA}, 108(22):8984--8989, 2011.

\bibitem[KW13]{KW3}
Y.~Kodama and L.~Williams.
\newblock The full {K}ostant-{T}oda hierarchy on the positive flag variety,
  2013.
\newblock Preprint, {\tt arXiv:1308.5011}, to appear in \emph{Comm. Math.
  Phys.}

\bibitem[Lus94]{Lusztig3}
G.~Lusztig.
\newblock Total positivity in reductive groups.
\newblock In {\em Lie theory and geometry}, volume 123 of {\em Progr. Math.},
  pages 531--568. Birkh\"auser Boston, Boston, MA, 1994.

\bibitem[Pos06]{Postnikov}
Alexander Postnikov.
\newblock Total positivity, grassmannians, and networks, 2006.
\newblock Preprint, {\tt arXiv:math/0609764}.

\bibitem[Pro82]{Proctor}
Robert~A. Proctor.
\newblock Classical {B}ruhat orders and lexicographic shellability.
\newblock {\em J. Algebra}, 77(1):104--126, 1982.

\bibitem[Rie99]{Rietsch1}
Konstanze Rietsch.
\newblock An algebraic cell decomposition of the nonnegative part of a flag
  variety.
\newblock {\em J. Algebra}, 213(1):144--154, 1999.

\bibitem[Sco06]{Scott}
Joshua~S. Scott.
\newblock Grassmannians and cluster algebras.
\newblock {\em Proc. London Math. Soc. (3)}, 92(2):345--380, 2006.

\bibitem[TW14]{TW}
Emmanuel Tsukerman and Lauren Williams.
\newblock Bruhat interval polytopes, 2014.
\newblock Preprint, {\tt arXiv:1406.5202}.

\bibitem[Wil07]{Williams}
Lauren~K. Williams.
\newblock Shelling totally nonnegative flag varieties.
\newblock {\em J. Reine Angew. Math.}, 609:1--21, 2007.

\end{thebibliography}



\end{document}